\documentclass[11pt]{amsart}
\usepackage{graphicx,amscd,color}

\theoremstyle{plain}
\newtheorem{theorem}{Theorem}[section]

\newtheorem{lemma}[theorem]{Lemma}

\newtheorem{definition}[theorem]{Definition}

\theoremstyle{remark}

\newtheorem{claim}{Claim}

\title [The topological minimality preserved by perturbation]
{Bridge surfaces with the topological minimality preserved by perturbation}

\author[J. H. Lee]{Jung Hoon Lee}
\address{Department of Mathematics and Institute of Pure and Applied Mathematics,
Chonbuk National University, Jeonju 54896, Korea}
\email{junghoon@jbnu.ac.kr}

\begin{document}

\begin{abstract}
We show that except for $n = 2$ if a bridge surface for a knot is an index $n$ topologically minimal surface,
then after a perturbation it is still topologically minimal with index at most $n+1$.
\end{abstract}

\maketitle

\section{Introduction}\label{sec1}

For a closed $3$-manifold $M$,
a {\em Heegaard splitting} $M = V^+ \cup_S V^-$ is a decomposition of $M$ into two handlebodies $V^+$ and $V^-$
with $\partial V^+ = \partial V^- = S$.

Let $K$ be a knot in $M$.
The notion of Heegaard splitting can be extended to the pair $(M, K)$.
Suppose that $V^{\pm} \cap K$ is a collection of $n$ boundary parallel arcs $a^{\pm}_1, \ldots, a^{\pm}_n$ in $V^{\pm}$.
Each $a^{\pm}_i$ is called a {\em bridge}.
The decomposition $(M, K) = (V^+, V^+ \cap K) \cup_S (V^-, V^- \cap K)$ is called
a {\em bridge splitting} of $(M, K)$, and
we say that $K$ is in $n$-bridge position with respect to $S$.
By a {\em bridge surface}, we mean $S - K$.

Compressing disks for the bridge surface $S - K$ in $M - K$ and the information on how they intersect
enable us to understand topological properties of $(M, K)$.
The {\em disk complex} $\mathcal{D}(F)$ of a surface $F$ embedded in a $3$-manifold
is a simplicial complex defined as follows.

\begin{itemize}
\item Vertices of $\mathcal{D}(F)$ are isotopy classes of compressing disks for $F$.
\item A collection of $k+1$ vertices forms a $k$-simplex
if there are representatives for each vertex that are pairwise disjoint.
\end{itemize}

A surface is {\em incompressible} if there are no compressing disks,
so the disk complex of an incompressible surface is empty.
A surface $F$ is {\em strongly irreducible} if $F$ compresses to both sides and
every compressing disk for $F$ on one side intersects every compressing disk on the opposite side.
So the disk complex of a strongly irreducible surface is disconnected.
Extending these notions, Bachman defined topologically minimal surfaces \cite{Bachman},
which can be regarded as topological analogues of (geometrically) minimal surfaces.

A surface $F$ is {\em topologically minimal} if
$\mathcal{D}(F)$ is empty or $\pi_i(\mathcal{D}(F))$ is non-trivial for some $i$.
The {\em topological index} of $F$ is $0$ if $\mathcal{D}(F)$ is empty, and
the smallest $n$ such that $\pi_{n-1}(\mathcal{D}(F))$ is non-trivial, otherwise.
Equivalently, an index $n$ topologically minimal surface $F$ has an $(n-2)$-connected $\mathcal{D}(F)$ and
$\pi_{n-1}(\mathcal{D}(F))$ is non-trivial.
Topologically minimal surfaces have nice properties, e.g.
if an irreducible manifold contains an incompressible surface and a topologically minimal surface, then
the two surfaces can be isotoped so that any intersection loop is essential in both surfaces.

A {\em perturbation} is an operation on a bridge splitting that perturbs $K$ near a point of $K \cap S$
so that a new local minimum and an adjacent local maximum is created.
The two new bridges admit cancelling disks that intersect in one point.
See Figure \ref{fig1}.
In this paper, we show that if a bridge surface is topologically minimal,
then a perturbation preserves topological minimality, except for one case.
More precisely,

\begin{theorem}\label{thm}
If a bridge surface for a knot is an index $n$($\ne 2$) topologically minimal surface,
then after a perturbation it is still topologically minimal with index at most $n+1$.
\end{theorem}

The main idea of the proof is to construct a retraction from the disk complex of a bridge surface to
a space whose homotopy group is non-trivial as in \cite{Bachman-Johnson} and \cite{Lee}.
We conjecture that the topological index of the perturbed bridge surface in Theorem \ref{thm} is $n+1$.

If we use (reduced) homology groups instead of homotopy groups in the definition of topological index,
then Theorem \ref{thm} holds for all $n$.

\begin{definition}
A surface $F$ is {\em strongly topologically minimal} if
$\mathcal{D}(F)$ is empty or $\widetilde{H}_i(\mathcal{D}(F))$ is non-trivial for some $i$.
The {\em strong topological index} of $F$ is $0$ if $\mathcal{D}(F)$ is empty, and
the smallest $n$ such that $\widetilde{H}_{n-1}(\mathcal{D}(F))$ is non-trivial, otherwise.
\end{definition}

\begin{theorem}\label{thm2}
If a bridge surface for a knot is an index $n$ strongly topologically minimal surface,
then after a perturbation it is still strongly topologically minimal with index at most $n+1$.
\end{theorem}

\section{Proof of Theorem \ref{thm}}\label{Sec2}

Let $M$ be decomposed into two handlebodies $V^+$ and $V^-$ with common boundary $S$, and
let $K$ be a knot in bridge position with respect to $S$.
Let $\overline{K}$ be a knot obtained from $K$ by a perturbation
with cancelling disks $D$ and $E$ in $V^+$ and $V^-$ respectively.
See Figure \ref{fig1}.
Let $\overline{D}$ denote a disk in $V^+ - K$ such that $\partial \overline{D} = \partial N(D \cap S)$, where
$N(D \cap S)$ is a neighborhood of $D \cap S$ taken in $S$, and
similarly let $\overline{E}$ denote a disk in $V^- - K$ such that $\partial \overline{E} = \partial N(E \cap S)$.
See Figure \ref{fig2}.
Obviously, we may assume that a disk disjoint from $D$ (resp. $E$)
is also disjoint from $\overline{D}$ (resp. $\overline{E}$).

\begin{figure}[ht!]
\begin{center}
\includegraphics[width=12cm]{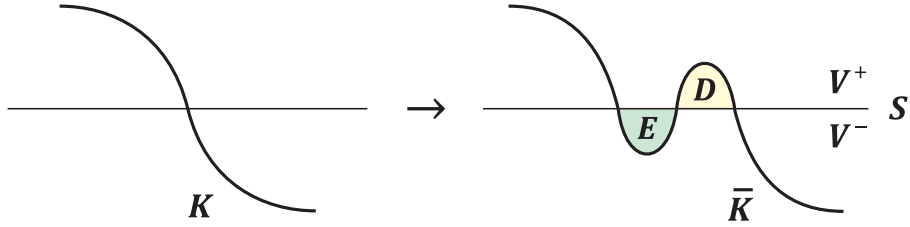}
\caption{A perturbation}\label{fig1}
\end{center}
\end{figure}

\begin{figure}[ht!]
\begin{center}
\includegraphics[width=6cm]{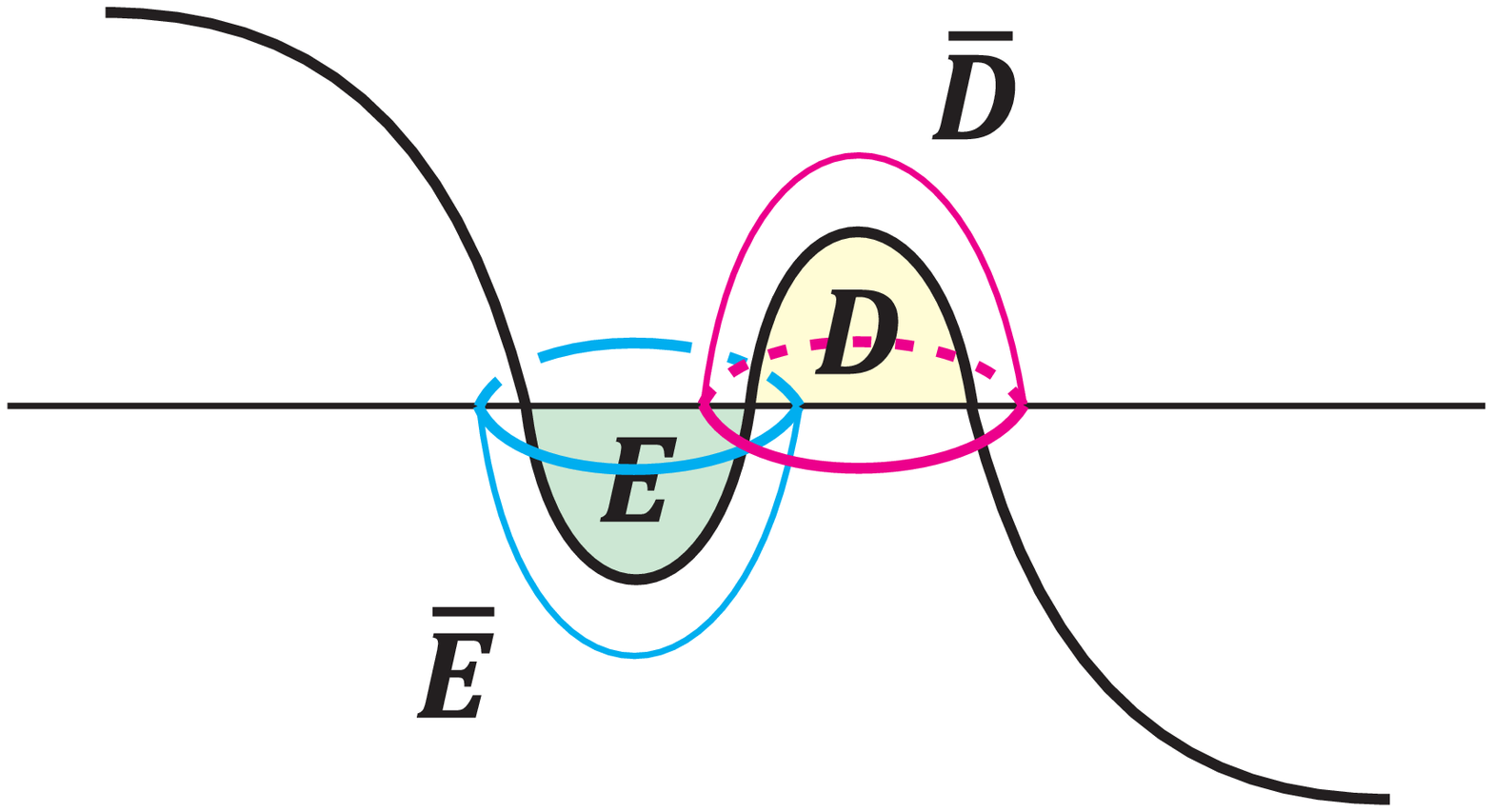}
\caption{$\overline{D}$ and $\overline{E}$}\label{fig2}
\end{center}
\end{figure}

\begin{lemma}\label{lem1}
We can naturally embed $\mathcal{D}(S - K)$ into $\mathcal{D}(S - \overline{K})$.
\end{lemma}

\begin{figure}[ht!]
\begin{center}
\includegraphics[width=12cm]{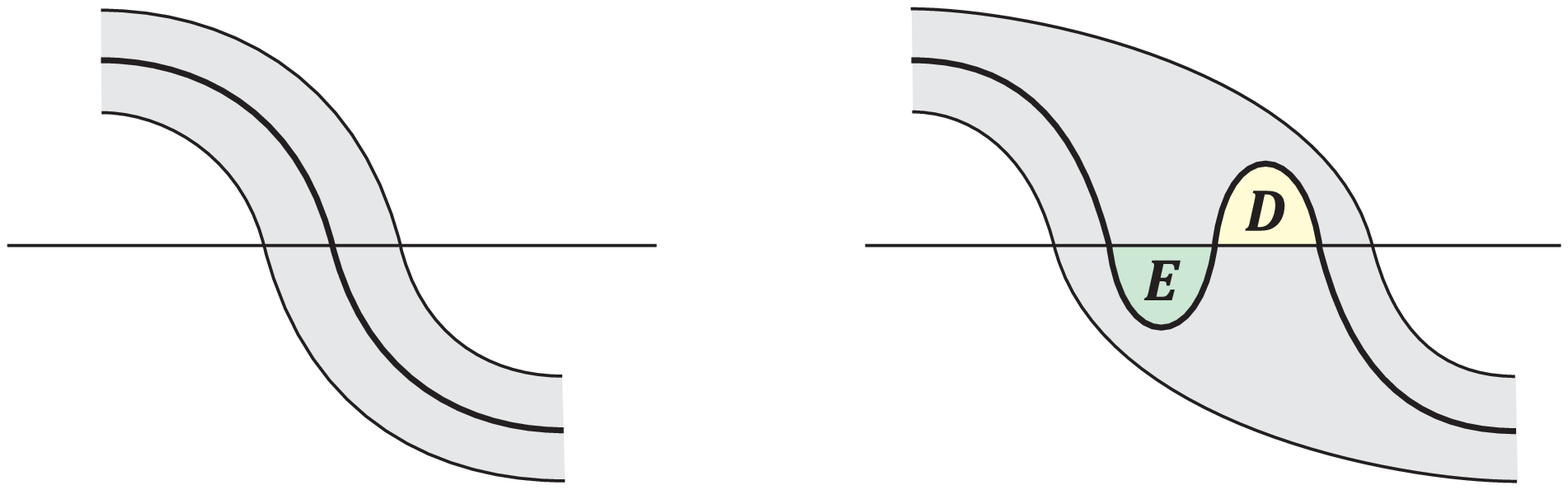}
\caption{$N(K)$ and $N(\overline{K} \cup D \cup E)$}\label{fig3}
\end{center}
\end{figure}

\begin{proof}
We can identify a neighborhood $N(K)$ of $K$
with a neighborhood $N(\overline{K} \cup D \cup E)$ of $\overline{K} \cup D \cup E$ since $|D \cap E| = 1$.
See Figure \ref{fig3}.
Then a compressing disk $C$ for $S - K$ in $V^{\pm} - N(K)$ corresponds to
a compressing disk $C'$ in $V^{\pm} - N(\overline{K} \cup D \cup E)$, and
$C'$ is also a compressing disk for $S -\overline{K}$ in $V^{\pm} - N(\overline{K})$.
Hence, by this embedding we may regard $\mathcal{D}(S - K)$ as a subcomplex of $\mathcal{D}(S - \overline{K})$.
\end{proof}

We give a partition of the set of vertices of $\mathcal{D}(S - \overline{K})$ as follows.

\begin{enumerate}
\item $\mathcal{E}_1 = \{ \overline{E} \}$
\item $\mathcal{E}_2 = \{$compressing disks in $V^- - \overline{K}$ other than $\overline{E} \}$
\item $\mathcal{E}_3 = \{$compressing disks in $V^+ - \overline{K}$ that are disjoint from $E \}$
\item $\mathcal{E}_4 = \{$compressing disks in $V^+ - \overline{K}$ that intersect $E \}$
\end{enumerate}

The four collections $\mathcal{E}_1, \mathcal{E}_2, \mathcal{E}_3, \mathcal{E}_4$ are mutually disjoint and
any compressing disk belongs to one of the collections.
A compressing disk of $\mathcal{D}(S - K)$ (as a subcomplex of $\mathcal{D}(S - \overline{K})$ by Lemma \ref{lem1})
in $V^- - \overline{K}$ (resp. $V^+ - \overline{K}$) belongs to $\mathcal{E}_2$ (resp. $\mathcal{E}_3$).
The disk $\overline{D}$ belongs to  $\mathcal{E}_4$.

We define a map $r_0$ from the set of vertices of $\mathcal{D}(S - \overline{K})$ to
the union of the set of vertices of $\mathcal{D}(S - K)$ and $\{\overline{D}, \overline{E}\}$.
Since $\overline{D}$ and $\overline{E}$ are disjoint from the compressing disks of $\mathcal{D}(S - K)$,
we can consider a simplicial complex ${\mathrm{Sus}}(\mathcal{D}(S - K))$,
which is the suspension of $\mathcal{D}(S - K)$ over $\{\overline{D}, \overline{E}\}$.
It will be shown later that the map $r_0$ extends to a retraction of
$\mathcal{D}(S - \overline{K})$ onto ${\mathrm{Sus}}(\mathcal{D}(S - K))$.

\vspace{0.1cm}

$(1)$ We define $r_0(\overline{E})$ to be $\overline{E}$.

\vspace{0.1cm}

$(2)$ Let $C \subset V^- - \overline{K}$ be a compressing disk other than $\overline{E}$.
Suppose that $C$ has nonempty minimal intersection (in its isotopy class) with $E$.
We may assume that $C \cap E$ consists of arc components by standard innermost disk argument.
Choose any outermost arc of $C \cap E$ in $C$ and let $\Delta$ be the corresponding outermost disk.
A disk surgery of $E$ along $\Delta$ yields a bridge disk and a compressing disk $C_1$ disjoint from $E$.
In case that $C$ does not intersect $E$, let $C_1 = C$.
So in any case, $C_1 \cap E = \emptyset$.

If there is any intersection point of $\partial C_1 \cap D$,
let $q$ be the point of $\partial C_1 \cap D$
which is closest to $p = D \cap E$ in the arc $D \cap S$.
Then we connect a copy of $\overline{E}$ to $C_1$ by a band along $\overline{pq}$ as in Figure \ref{fig4} and
get a new disk $C_2$ with $|\partial C_2 \cap D| < |\partial C_1 \cap D|$.
We perform this operation for all intersection points of $\partial C_1 \cap D$, and
let $C'$ be the resulting disk with $\partial C' \cap D = \emptyset$.
In fact, $C'$ is isotopic to $C_1$ if we remove $E$ by isotopy as in Figure \ref{fig5}.
We define $r_0(C)$ to be $C'$.
Note that $C' \cap (D \cup E) = \emptyset$, hence
$C'$ can be regarded as a disk in $\mathcal{D}(S - K)$.

\begin{figure}[ht!]
\begin{center}
\includegraphics[width=8cm]{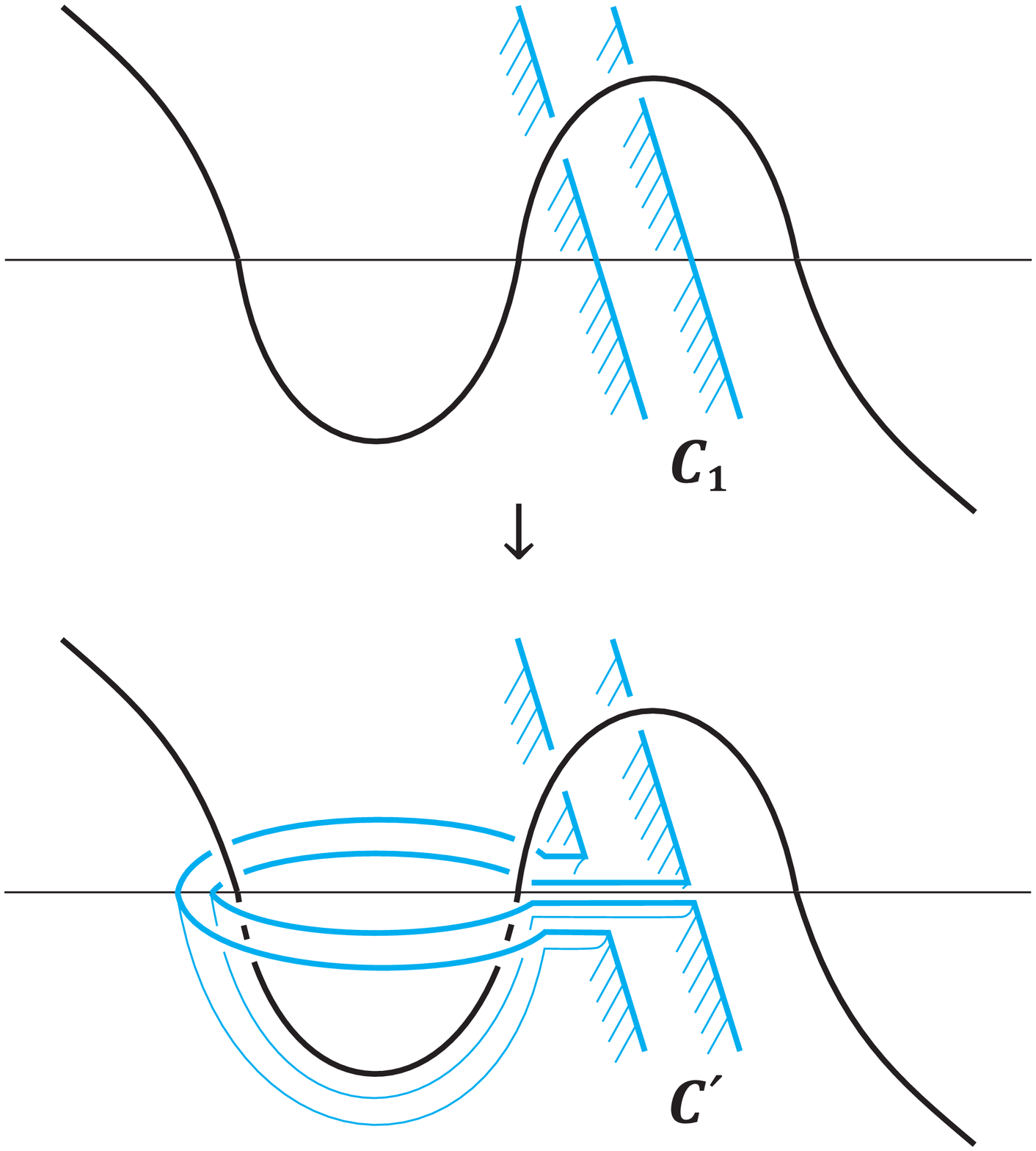}
\caption{}\label{fig4}
\end{center}
\end{figure}

\begin{figure}[ht!]
\begin{center}
\includegraphics[width=8cm]{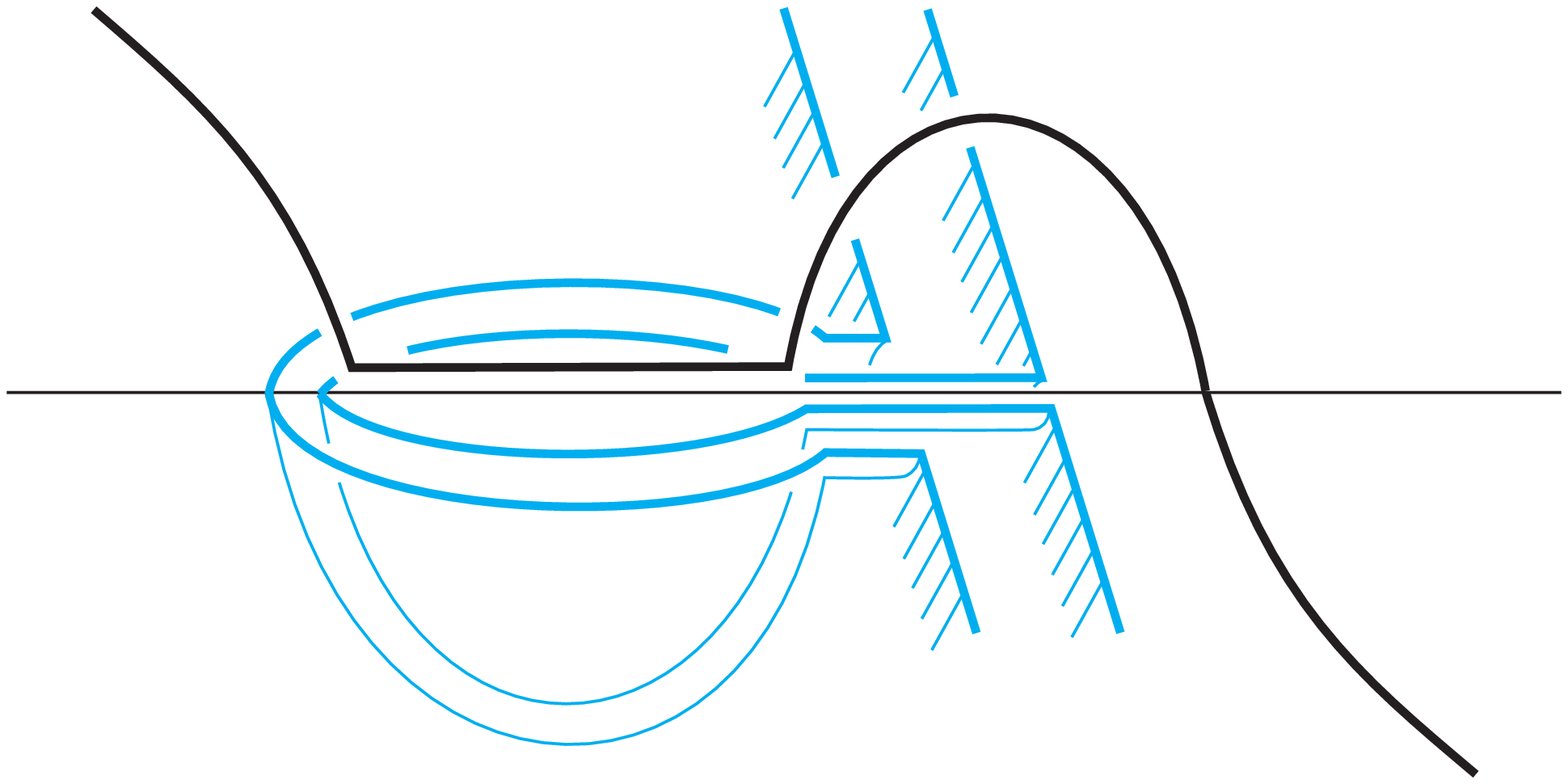}
\caption{}\label{fig5}
\end{center}
\end{figure}

\vspace{0.1cm}

$(3)$ Let $C \subset V^+ - \overline{K}$ be a compressing disk that is disjoint from $E$.
Suppose that $C$ intersects $D$.
We may assume that $C \cap D$ consists of arc components.
For every arc $\alpha$ of $C \cap D$, we cut off $C$ by $\alpha$ and
reglue the two resulting subdisks along a slightly detouring band
passing through the bridge disk($\ne D$) adjacent to $E$, as in Figure \ref{fig6}.
Even though arcs of $C \cap D$ are nested in $D$, this operation is possible.
Let $C'$ be the resulting disk obtained from $C$.
If $C$ does not intersect $D$, let $C' = C$.
The disk $C'$ is isotopic to $C$ if we remove $E$ by isotopy.
We define $r_0(C)$ to be $C'$.
Note that $C' \cap (D \cup E) = \emptyset$,
hence $C'$ can be regarded as a disk in $\mathcal{D}(S - K)$.

\begin{figure}[ht!]
\begin{center}
\includegraphics[width=8cm]{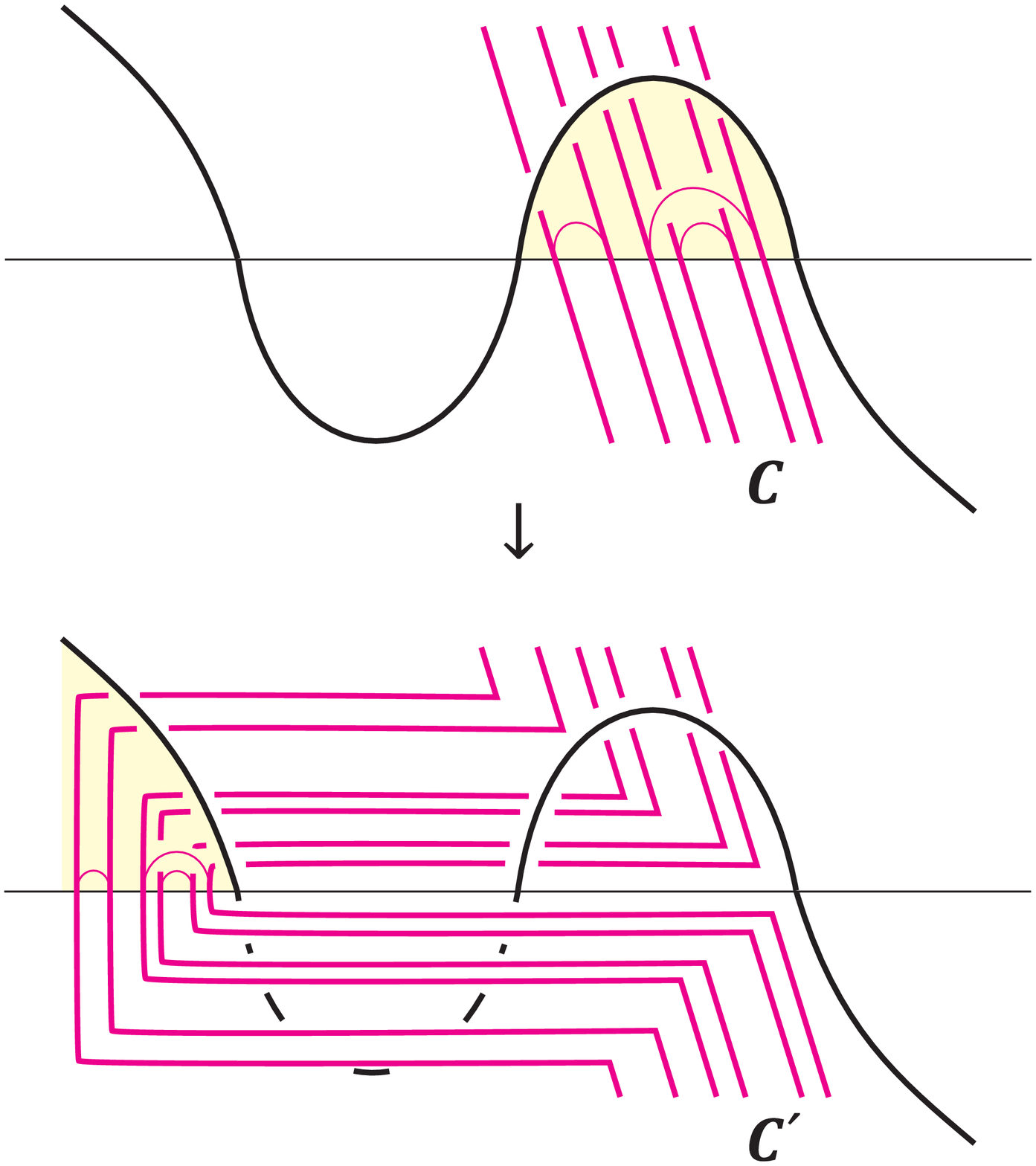}
\caption{}\label{fig6}
\end{center}
\end{figure}

\vspace{0.1cm}

$(4)$ Let $C \subset V^+ - \overline{K}$ be a compressing disk that intersects $E$.
We define $r_0(C)$ to be $\overline{D}$.

The restriction of $r_0$ to the set of vertices of ${\mathrm{Sus}}(\mathcal{D}(S - K))$ is the identity map.
Next we show that $r_0$ can be extended to a continuous map $r_1$
from the $1$-skeleton of $\mathcal{D}(S - \overline{K})$ to the $1$-skeleton of ${\mathrm{Sus}}(\mathcal{D}(S - K))$.
Then $r_1$ will be a retraction.

\begin{lemma}\label{lem2}
The map $r_0$ extends to a continuous map $r_1$
from the $1$-skeleton of $\mathcal{D}(S - \overline{K})$ to the $1$-skeleton of ${\mathrm{Sus}}(\mathcal{D}(S - K))$.
\end{lemma}

\begin{proof}
It suffices to show that for disjoint compressing disks $C_1$ and $C_2$ of $\mathcal{D}(S - \overline{K})$,
either $r_0(C_1)$ and $r_0(C_2)$ are disjoint, or $r_0(C_1) = r_0(C_2)$.
Without loss of generality, there are the following cases to consider.

\vspace{0.1cm}

Case $1$. $C_1 \in \mathcal{E}_1$ and $C_2 \in \mathcal{E}_2$.

Since $r_0(C_1) = r_0(\overline{E}) = \overline{E}$ and $r_0(C_2)$ is a disk in $\mathcal{D}(S - K)$,
$r_0(C_1)$ and $r_0(C_2)$ are disjoint.

\vspace{0.1cm}

Case $2$. $C_1 \in \mathcal{E}_1$ and $C_2 \in \mathcal{E}_3$.

The disk $r_0(C_2)$ is a disk in $\mathcal{D}(S - K)$.
So similarly as Case $1$, $r_0(C_1)$ and $r_0(C_2)$ are disjoint.

\vspace{0.1cm}

Case $3$. $C_1 \in \mathcal{E}_2$ and $C_2 \in \mathcal{E}_3$.

Suppose that $C_1$ intersects $E$.
The operations in the definition of $r_0(C_1)$ are done in two steps.
In the step of disk surgery of $E$ along outermost disk,
the resulting disk is disjoint from $C_2$ because $C_2$ is disjoint from $E$.

The remaining banding operations for $r_0(C_1)$ and $r_0(C_2)$ result disjoint $r_0(C_1)$ and $r_0(C_2)$.
See Figure \ref{fig7} for an example.

\begin{figure}[ht!]
\begin{center}
\includegraphics[width=8cm]{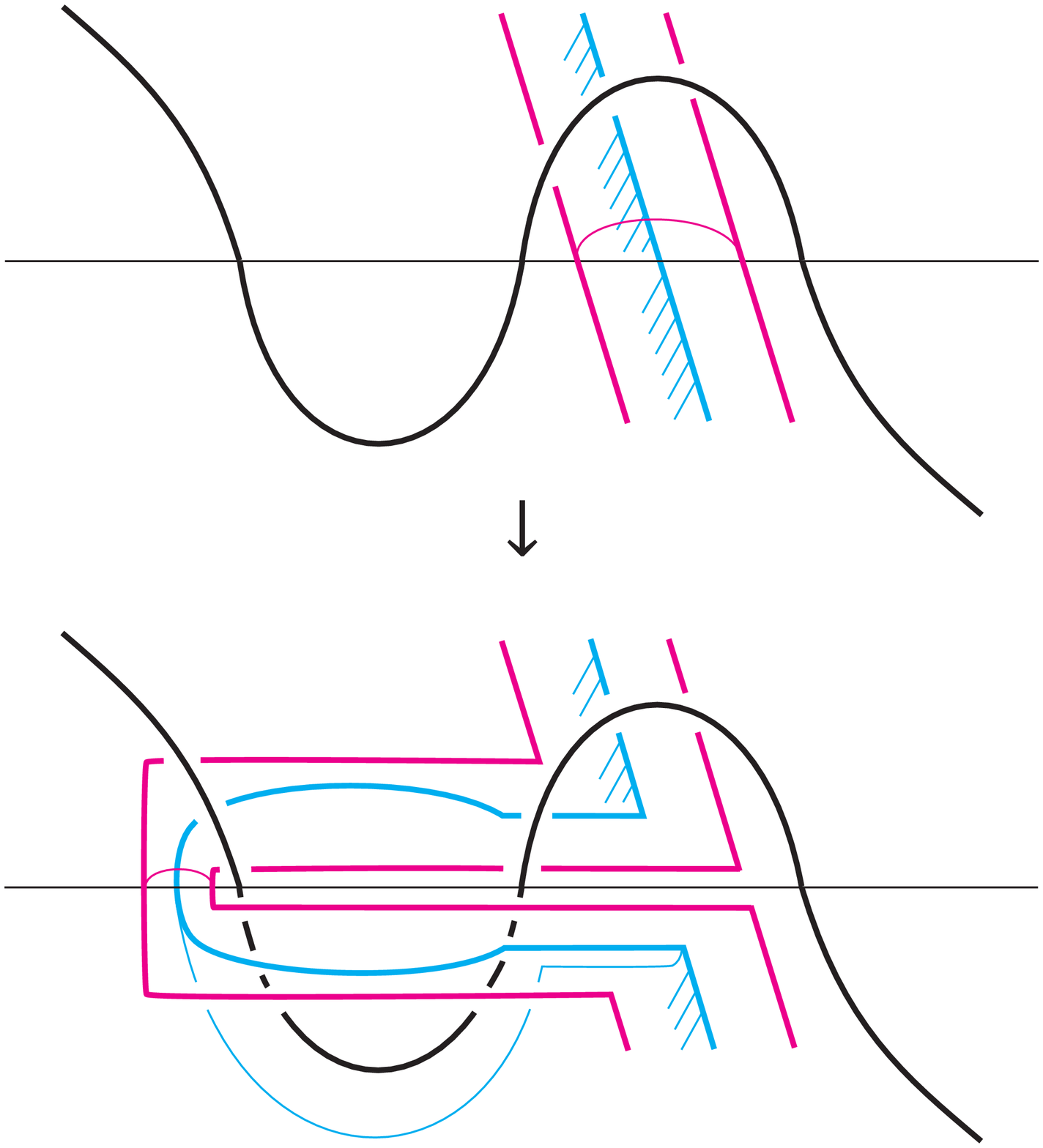}
\caption{}\label{fig7}
\end{center}
\end{figure}

\vspace{0.1cm}

Case $4$. $C_1 \in \mathcal{E}_2$ and $C_2 \in \mathcal{E}_4$.

Since $r_0(C_1)$ is a disk in $\mathcal{D}(S - K)$ and $r_0(C_2) = \overline{D}$,
$r_0(C_1)$ and $r_0(C_2)$ are disjoint.

\vspace{0.1cm}

Case $5$. $C_1 \in \mathcal{E}_3$ and $C_2 \in \mathcal{E}_4$.

Since $r_0(C_1)$ is a disk in $\mathcal{D}(S - K)$ and $r_0(C_2) = \overline{D}$,
$r_0(C_1)$ and $r_0(C_2)$ are disjoint.

\vspace{0.1cm}

Case $6$. Both $C_1, C_2 \in \mathcal{E}_2$.

We can see that both the disk surgery and banding operations of $r_0$ result
disjoint $r_0(C_1)$ and $r_0(C_2)$, or $r_0(C_1) = r_0(C_2)$.

\vspace{0.1cm}

Case $7$. Both $C_1, C_2 \in \mathcal{E}_3$.

We can see that $r_0(C_1)$ and $r_0(C_2)$ are disjoint, or $r_0(C_1) = r_0(C_2)$.

\vspace{0.1cm}

Case $8$. Both $C_1, C_2 \in \mathcal{E}_4$.

In this case, $r_0(C_1) = r_0(C_2) = \overline{D}$.
\end{proof}

Since higher dimensional simplices of $\mathcal{D}(S - \overline{K})$ are determined by $1$-simplices,
$r_1$ extends to a retraction $r : \mathcal{D}(S - \overline{K}) \to {\mathrm{Sus}}(\mathcal{D}(S - K))$.

\vspace{0.1cm}

Suppose that $S - K$ is an index $n$ topologically minimal surface.
First, consider the case of $n = 0$.
Then the incompressibility of $S - K$ implies that the genus of $S$ is $0$ and $K$ is in $1$-bridge position,
i.e. $K$ is a $1$-bridge unknot in $S^3$.
A perturbation of $K$ yields a $2$-bridge splitting for the unknot,
which is strongly irreducible, hence index $1$.
So Theorem \ref{thm} holds when $n = 0$.
Now we assume that $n = 1$ or $n \ge 3$.

\begin{claim}\label{claim}
$\pi_n({\mathrm{Sus}}(\mathcal{D}(S - K))) \ne 1$.
\end{claim}

\begin{proof}[Proof of Claim \ref{claim}]
Suppose that $n = 1$.
Since $S - K$ is an index $1$ topologically minimal surface,
$\pi_0(\mathcal{D}(S - K)) \ne 1$, i.e. $\mathcal{D}(S - K)$ is disconnected.
In fact, it has two contractible components---
the subcomplexes spanned by compressing disks in $V^+ - K$ and $V^- - K$.
Then the fundamental group of the suspension of $\mathcal{D}(S - K)$ is infinite cyclic and the claim holds.

So we may assume that $n \ge 3$.
Since $S - K$ is an index $n$ topologically minimal surface,
$\mathcal{D}(S - K)$ is $(n-2)$-connected and $\pi_{n-1}(\mathcal{D}(S - K)) \ne 1$.
It is known that the suspension map
$$\pi_{i}(\mathcal{D}(S - K)) \to \pi_{i+1}({\mathrm{Sus}}(\mathcal{D}(S - K)))$$
is an isomorphism for $i < 2(n-1) - 1$ and a surjection for $i = 2(n-1) - 1$.
(See e.g. \cite[Corollary 4.24]{Hatcher}.)
Hence $\pi_n({\mathrm{Sus}}(\mathcal{D}(S - K))) \ne 1$.
\end{proof}

The retraction $r$ induces a surjective map
$r_{\ast} : \pi_n(\mathcal{D}(S - \overline{K})) \to \pi_n({\mathrm{Sus}}(\mathcal{D}(S - K)))$.
So $\pi_n(\mathcal{D}(S - \overline{K})) \ne 1$, and the topological index of $S - \overline{K}$ is at most $n+1$.

\section{When $n = 2$ and the non-trivial homology condition}\label{sec3}

In this section we investigate the case of $n = 2$ in detail.
Suppose that $S - K$ is an index $2$ topologically minimal surface.
Then $\pi_0(\mathcal{D}(S - K)) = 1$ and $\pi_1(\mathcal{D}(S - K)) \ne 1$.
The suspension map $\pi_{1}(\mathcal{D}(S - K)) \to \pi_{2}({\mathrm{Sus}}(\mathcal{D}(S - K)))$
is a surjection and does not guarantee that $\pi_{2}({\mathrm{Sus}}(\mathcal{D}(S - K)))$ is non-trivial.

The suspension ${\mathrm{Sus}}(\mathcal{D}(S - K))$ is
a union of an upper cone $A$ and a lower cone $B$ and $A \cap B \simeq \mathcal{D}(S - K)$.
By the van Kampen theorem, $\pi_1({\mathrm{Sus}}(\mathcal{D}(S - K))) = 1$
because $\pi_1(A) = \pi_1(B) = 1$.
Applying the Mayer-Vietoris sequence, the long exact sequence
$$\cdots \to \widetilde{H}_{i+1}(A) \oplus \widetilde{H}_{i+1}(B) \to
\widetilde{H}_{i+1}({\mathrm{Sus}}(\mathcal{D}(S - K))) \to
\widetilde{H}_i(\mathcal{D}(S - K)) \to
\widetilde{H}_i(A) \oplus \widetilde{H}_i(B) \to \cdots$$
is exact.
Since $\widetilde{H}_i(A) = \widetilde{H}_i(B) = 1$ for all $i$, we have
$$\widetilde{H}_{i+1}({\mathrm{Sus}}(\mathcal{D}(S - K))) \simeq \widetilde{H}_i(\mathcal{D}(S - K)).$$
In particular, $\widetilde{H}_2({\mathrm{Sus}}(\mathcal{D}(S - K))) \simeq \widetilde{H}_1(\mathcal{D}(S - K))$.
Since $\pi_0({\mathrm{Sus}}(\mathcal{D}(S - K))) = 1$ and
$\pi_1({\mathrm{Sus}}(\mathcal{D}(S - K))) = 1$, by the Hurewicz theorem
$\pi_2({\mathrm{Sus}}(\mathcal{D}(S - K))) \simeq \widetilde{H}_2({\mathrm{Sus}}(\mathcal{D}(S - K)))$.
So we conclude that
$\pi_2({\mathrm{Sus}}(\mathcal{D}(S - K))) \simeq \widetilde{H}_1(\mathcal{D}(S - K))$.
The group $\widetilde{H}_1(\mathcal{D}(S - K))$ is isomorphic to
$\pi_1(\mathcal{D}(S - K)) / C$, where $C$ is the commutator subgroup of $\pi_1(\mathcal{D}(S - K))$.
Hence if $\pi_1(\mathcal{D}(S - K))$ is not equal to $C$, i.e.
if $\pi_1(\mathcal{D}(S - K))$ is not a perfect group,
then $\pi_2({\mathrm{Sus}}(\mathcal{D}(S - K))) \ne 1$ and
the topological index of $S - \overline{K}$ is at most $3$.

\begin{proof}[Proof of Theorem \ref{thm2}]
Suppose that $S - K$ is an index $n$ strongly topologically minimal surface.
The case of $n=0$ is similar to the proof of Theorem \ref{thm}.
So we assume that $n \ge 1$.
By definition, $\widetilde{H}_{n-1}(\mathcal{D}(S - K)) \ne 1$.
The above mentioned isomorphism
$\widetilde{H}_n({\mathrm{Sus}}(\mathcal{D}(S - K))) \to \widetilde{H}_{n-1}(\mathcal{D}(S - K))$
implies that $\widetilde{H}_n({\mathrm{Sus}}(\mathcal{D}(S - K))) \ne 1$.
The retraction $r$ induces a surjective map
$r_{\ast} : \widetilde{H}_n(\mathcal{D}(S - \overline{K})) \to \widetilde{H}_n({\mathrm{Sus}}(\mathcal{D}(S - K)))$.
So $\widetilde{H}_n(\mathcal{D}(S - \overline{K})) \ne 1$, and
 the strong topological index of $S - \overline{K}$ is at most $n+1$.
\end{proof}

\noindent {\bf Acknowledgements.}
The author would like to thank Jae Choon Cha and Daewoong Lee for helpful comments.


\end{document}